\documentclass{amsart}
\usepackage{amsmath, amsthm, amssymb}

\newcommand{\id}{\mathrm{id}}
\newcommand{\isom}{\cong}

\newcommand{\sm}{\setminus}

\newcommand{\C}{\mathbb{C}}

\newcommand{\Q}{\mathbb{Q}}

\newcommand{\Z}{\mathbb{Z}}

\newcommand{\im}{\operatorname{im}}

\newcommand{\ip}[2]{\left< #1 , #2 \right>}

\newtheorem{thm}{Theorem}
\newtheorem{lemma}[thm]{Lemma}
\newtheorem{theorem}[thm]{Theorem}
\newtheorem{prop}[thm]{Proposition}
\newtheorem{cor}{Corollary}[thm]

\newcommand{\Ext}{\operatorname{Ext}}

\newcommand{\EA}{\Lambda}
\newcommand{\rank}{\operatorname{rank}}

\newcommand{\Gl}{\operatorname{Gl}}

\newcommand{\diag}{\operatorname{diag}}

\renewcommand{\d}[1]{\widehat{#1}}

\newcommand{\Hom}{\operatorname{Hom}}

\newcommand{\tensor}{\otimes}

\newcommand{\pt}{\text{\rm pt}}

\renewcommand{\P}{\mathbb{P}}

\newcommand{\heq}{\overset{\text{\rm h.e.}}{\simeq}}

\newcommand{\weq}{\simeq}

\input{xypic}
\xyoption{all}

\begin{document}
\title{The $S^1$-Equivariant Cohomology of Spaces of Long Exact Sequences} \author{T.B. Williams}
\begin{abstract}Let $S$ denote the graded polynomial ring $\C[x_1,\dots,x_m]$. We interpret a chain
  complex of free $S$-modules having finite length homology modules as an $S^1$-equivariant map $\C^m\sm\{0\} \to
  X$, where $X$ is a moduli space of exact sequences. By computing the cohomology of such spaces $X$
  we obtain obstructions to such maps, including a slight generalization of the Herzog-K\"uhl
  equations.
\end{abstract}
\maketitle
\tableofcontents

\section{Introduction}
Let $X(c_1, \dots, c_n)$ denote the space of long exact sequences
\begin{equation*}
  \xymatrix{ 0 \ar[r] & k^{c_1} \ar[r] & k^{c_2} \ar[r] & \dots \ar[r] & k^{c_3} \ar[r] & 0 }
\end{equation*}
which we handle by fixing bases for each term and identifying a sequence with the $n-1$-tuple of
matrices representing the differentials. The principal results we obtain are a calculation of the
cohomology of this space, and a description of the $E_\infty$-page of a spectral sequence converging
to the $S^1$-equivariant cohomology of $X$.

The eventual applications we have in mind are to the conjecture of Buchsbaum-Eisenbud \& Horrocks
and the conjecture of Carlsson in homological algebra. The first of these conjectures, in a special
case, claims that if $M$ is a nonzero finite-length module over the graded polynomial ring $S = \C[x_1,
\dots, x_m]$, and if
\begin{equation*}
  \xymatrix{ 0 \ar[r] & F_a \ar[r] & F_{a-1} \ar[r] & \dots \ar[r] & F_0 \ar[r] & M }
\end{equation*}
is a free resolution, then $\rank F_i \ge \binom{m}{i}$, or more informally, that the smallest
resolution is the Koszul resolution.

Decomposing the free modules into their graded parts: $F_i = \bigoplus_{j=0}^\infty
S(j)^{\beta_{i,j}}$, we obtain a matrix $\beta_{i,j}$ of integers, referred to as the \textit{Betti
  table\/} of the resolution.

Recently, in \cite{ERMAN}, some progress was made on the conjecture by means of the striking theory of
Boij-Soderberg \cite{BOIJ-SODERBERG}, \cite{EISENBUDSCHREYER}, that states that the Betti table of a resolution
of such a module is constrained to be a rational linear combination of certain simpler tables. The
starting point for the study of Betti tables are the following formulas, which go by the name of the
Herzog-K\"uhl equations.
\begin{equation*}
  \sum_{i=1}^a \sum_{j \in \Z} (-1)^j\beta_{i,j}j^s = 0
\end{equation*}
where $s$ ranges over the integers $\{0, \dots, n\}$. 

Parallel to Buchsbaum-Eisenbud \& Horrocks conjecture is a conjecture of Carlsson, \cite{CARLSSON3},
which is more general in its applicability but less specific in its claims. According to that conjecture, if
$M$ is a differential-graded free module over $S$, and if $M$ has nonzero finite-length homology, then it
is claimed that $\rank M \ge 2^m$. 

The simplest examples of differential-graded modules are the total spaces of resolutions:
$\bigoplus_{i=0}^a F_i$. Occupying an intermediate position between resolutions on the one hand, and
the full generality of \textsc{dgm}s on the other, there are the chain complexes of graded free
modules
\begin{equation*}
  \Theta: \,\xymatrix{ 0 \ar[r] & F_1 \ar[r] & F_{2} \ar[r] & \dots \ar[r] & F_{a} \ar[r] & 0 }
\end{equation*}
having finite length homology.

The differential of $\Theta$ can be written as a matrix $D:\bigoplus_i F_i \to \bigoplus_i F_i$,
taking entries in the ring $S$. The graded nature of $D$ and $\bigoplus_i F_i$ ensures that the
polynomials appearing in $D$ are homogeneous, although they may be of varying degrees. We view such
a matrix as a map $\C^m \to \operatorname{Mat}_{N \times N}(\C)$, for the appropriate value of
$N$. A basic result, c.f.~\cite{CARLSSON2}, is that the homology of $\Theta$ is an $S$-module which
is supported only at the ideal $(x_1, \dots, x_m)$. In this case, this implies that evaluating
$\Theta$ at any point $x \in \C^m \sm \{0\}$ yields a long exact sequence. We have a map
\begin{equation*}
  D: \C^m \sm \{0\} \longrightarrow X
\end{equation*}
where $X$ is the moduli space of long exact sequences considered above. Since the entries in $D$ are
homogeneous, this map is compatible with a $\C^* \weq S^1$-action given by scalar multiplication. By
considering equivariant cohomology with respect to this $\C^*$-action, we are able to obtain
obstructions to the existence of $D$. These obstructions include, most notably, the Herzog-K\"uhl
equations, generalized to this case. There are also some other obstructions in a similar vein,
which, if the Buchsbaum-Eisenbud-Horrocks conjecture holds, are conjecturally vacuous in the case of
resolution, but may still have relevance in the more general case.

The usual derivation of the Herzog-K\"uhl equations for resolutions is by consideration of the
Hilbert polynomial of the module $M$ being resolved. The Hilbert polynomial, understood here as
defined on $K_0(\P^{n-1})$ corresponds via a Riemann-Roch theorem to the the Chern polynomial of $M$
in $CH^*(\P^{n-1})$, \cite[p. 490]{EISENBUD}. It should come as no surprise that the obstructions we
obtain in theorem \ref{t:HKE} are exactly the Chern classes of $M$ in the case of a resolution. It is
instructive to see quite how little is required for the derivation of these obstructions, since they
rely only on the $\C^*$-equivariance of the map $\C^m \sm \{0\} \to X$; no other property of
homogeneous polynomials is used.

All equivariant homology and cohomology will be Borel equivariant homology and cohomology with
respect to the group $S^1 \subset \C^*$ of unit complex numbers, unless otherwise stated. We
abbreviate $\Gl(n,\C)$ to $\Gl(n)$ throughout.

This paper will be recast in the language of equivariant motivic cohomology in future work, thus
extending the result to the case of arbitrary characteristic. Since motivic cohomology theory is
bigraded, having a `weight' grading in addition to the usual one. This grading facilitates a number
of arguments about differentials in spectral sequences, which are forced to vanish for reasons of
grading. As a consequence, we leave in sketch form a handful of arguments concerning spectral
sequence, mostly to the effect that certain classes are transgressive. All such arguments can be
carried out without recourse to algebraic theories, generally by ungainly arguments considering the
action of cohomology operations on the spectral sequence.

\section{Graded Vector Spaces}
\label{SecGRVS}

We work exclusively over $\C$. We understand a \textit{graded vector space $V$} to be a vector space
equipped with a decomposition $V = \bigoplus_{i \in \Z} V_i$. The subspace $V_i$ is the $i$-th
graded part of $V$. Graded vector spaces can be realized as representations of $S^1$, by letting
$S^1$ act on $V_i$ by $z \circ v = z^iv$. When we consider a basis of a graded vector space, we will
generally require that such a basis consist of homogeneous elements.

Suppose $V$, $W$ are finite dimensional graded vector spaces, then there are two ways in which we may give
$\Hom_\C(V,W)$ an $S^1$-action.
\begin{enumerate}
\item The left action, so that $(z \circ T) v = z \circ (Tv)$.
\item The left-right action so that: $(z \circ T)(z \circ v) = z \circ (Tv)$.
\end{enumerate}
Both actions will occur in the sequel. Since the action of $S^1$ on a graded vector space maps bases
to bases, it is easy to see that $z \circ A$ has the same rank as $A$ in either case. If we
fix graded bases $\{x_1, \dots,x_n\}$, $\{y_1, \dots, y_m\}$ of $V,W$, and if these elements lie in
gradings $\{v_1, \dots, v_n\}$, $\{w_1, \dots, w_m\}$ respectively, then a given map $ V \to W$ can
be written in matrix form, $A$. The two actions of $z \in S^1 \subset \C^*$ on $A$ are both given by 
matrix multiplication
\begin{equation*}
  z \cdot A =\begin{pmatrix} z^{w_1} & & & \\ &  z^{w_2} & &  \\
    & & \ddots  & \\ & & & z^{w_m} \end{pmatrix} A \qquad \text{ in the case of action on the left}
\end{equation*}
and
\begin{equation*}
  z \cdot A =\begin{pmatrix} z^{w_1} & & & \\ &  z^{w_2} & &  \\
    & & \ddots  & \\ & & & z^{w_m} \end{pmatrix} A \begin{pmatrix} z^{-v_1} & & & \\ &  z^{-v_2} & &  \\
    & & \ddots  & \\ & & & z^{-v_n} \end{pmatrix}
\end{equation*}
 in the case of action on both the left and right.
    
Given a finite-dimensional graded vector space $V$, it has a dual space $\Hom_\C(V,\C)$. This can be
given an $S^1$ action, and consequently a grading, in either of the two ways above, but in the case
of vector spaces over $\C$ there is also a third construction. We continue in our choice of basis
$\{x_i\}$ of $V$, and we can define a sesquilinear form by $\ip{x_i}{x_j} = \delta_{ij}$
and extending to all of $V$; this gives an explicit isomorphism $V \isom \d V$. We make $\d V$ into
a graded vector space by means of this isomorphism. It is not hard to see that the grading on $\d V$
we obtain in this way does not depend on the chosen graded basis. Henceforth, whenever we take the
dual of a graded vector space, we shall give it this grading.

We continue to employ the same bases $\{x_i\}$ and $\{y_i\}$ for the graded vector spaces $V$, $W$
as before, and deliberately confuse transformations and matrices. If $A \in \Hom(V,W)$, then $A$ has
a Hermitian conjugate dual, $A'$, which is a map of vector spaces $A' \in \Hom(\d W, \d V)$. With
this notation, we can assert the following.
\begin{prop} \label{foldingequivprec}
  The map $\Hom(V,W) \to \Hom(\d W, \d V)$ given by $A \mapsto A'$ is $S^1$-equivariant
  when both spaces are given the left-right action.
\end{prop}
\begin{proof}
  This is most easily seen by arguing explicitly
  with matrices. Write $A = (a_{i,j})$. With respect to the left-right action on $A'$, we have
  \begin{equation*} \begin{split}
    (z \circ A' )( y_i) =  (z \circ A') ( z^{-w_j} ( z \circ y_i)) = z^{-w_j} (z \circ A')
    ( z \circ y_i) = \\ z^{-w_j} \Big[ z \circ \sum_{j=1}^n \overline a_{j,i}
    x_j \Big] =  z^{-w_j}\sum_{j=1}^n  \overline a_{j,i} z^{v_j} x_j \end{split}
  \end{equation*}
  as a result, we have
  \begin{equation*}
    z \circ A' =\begin{pmatrix} z^{v_1} & & & \\ &  z^{v_2} & &  \\
      & & \ddots  & \\ & & & z^{v_n} \end{pmatrix} A' \begin{pmatrix} z^{-w_1} & & & \\ &  z^{-v_2} & &  \\
      & & \ddots  & \\ & & & z^{-v_n} \end{pmatrix} = (z \circ A)',
  \end{equation*}
  as required.
\end{proof}

\section{The Homology and Cohomology of $\Gl(n)$}
\label{s:HoCoHoGln}

We take coefficients in a commutative ring $R$. In all applications $R$ shall be $\Z$, $\Z/p$ or
$\Q$.  The cohomology of $\Gl(n)$ is an exterior algebra $H^*(\Gl(n); R) = \EA_{R}(\alpha_1, \dots,
\alpha_n)$ with $|\alpha_i| = 2i-1$. The elements $\alpha_i$ are primitive, in the sense of Hopf
algebra. The homology is given a ring structure by the Pontrjagin product, which turns out to be an
exterior algebra again $H_*(\Gl(n);R) = \EA_{R}(\d{\alpha}_1, \dots, \d{\alpha}_n)$. Indeed one
has $\d{\beta}\d{\gamma} = \d{\beta \gamma}$ with respect to these two product structures.

There is a map $\Gl(n) \times \Gl(m) \to \Gl(n+m)$ by matrix direct-sum. The induced map on homology
is 
\begin{align*}\EA_{R}(\d{\alpha}_1, \dots, \d{\alpha}_n) \tensor \EA_{R}(\d{\alpha}'_1, \dots,
\d{\alpha}'_m) &\longrightarrow \EA_R(\d{\alpha}''_1, \dots, \d{\alpha}''_{n+m})\\
\d{\alpha}_i \tensor
\d{\alpha}_j' \longmapsto \d{\alpha}_i'' \d{\alpha}_j''\end{align*}

There are two noteworthy involutions $\Gl(n) \to \Gl(n)$. They are $\iota: a \mapsto a^{-1}$ and $c:
a \mapsto a'$. If we restrict to $U(n) \heq \Gl(n)$, they coincide. As a result, the automorphisms
they induce on homology and cohomology coincide. From the sequence of maps
\begin{equation*}
  \xymatrix{ \Gl(n) \ar^{\Delta}[r] & \Gl(n) \times \Gl(n) \ar[r]^{\iota \times \id} & \Gl(n) \times
    \Gl(n) \ar^{\mu}[r] & \Gl(n)}
\end{equation*}
we obtain a diagram on homology groups
\begin{equation*} 
  \xymatrix{ H_*(\Gl(n); R) \ar^{\Delta_*}[r] & H_*(\Gl(n) \times \Gl(n); R) \ar[r]^{\qquad \iota_*
      \tensor \id} & & \\
    & H_*(\Gl(n) \times \Gl(n); R) \ar[r] & H_*(\Gl(n); R) }
\end{equation*}
which, when applied to $\d\alpha_i$ gives
\begin{equation*}
  \xymatrix{ \d\alpha_i \ar@{|->}[r] & \d\alpha_i \tensor 1 + 1 \tensor \d\alpha_i \ar@{|->}[r] &
    \iota_*(\d\alpha_i) \tensor 1 + 1 \tensor \d\alpha_i \ar@{|->}[r] & c_*(\d\alpha_i) + \d\alpha_i}
\end{equation*}
since the composition is $\Gl(n) \to \{I_n\} \subset \Gl(n)$, we must have $c_*(\d\alpha) =
\iota_*(\d\alpha_i) = - \d\alpha_i$.

\section{The Equivariant Cohomology of $\Gl(n)$}

Given a space $X$ with a $G$ action, we call the Borel equivariant cohomology $H^*(EG \times_G X;
R)$ the \textit{equivariant cohomology}. There exists a fiber sequence $X \to EG \times_G X \to BG$,
which, under hypotheses which will always be satisfied in our applications, gives a Serre spectral
sequence
\begin{equation*}
  E_2^{p,q} = H^p( BG; H^q(X); R) \Longrightarrow H^{p+q}(EG \times_G X;R).
\end{equation*}
In general we shall not compute equivariant cohomology in full, but shall content ourselves to
describe the $E_2$-page of this spectral sequence and the differentials.

\subsection{The Left Action}

Let $\vec{w} = (w_1, \dots, w_n) \in \Z^n$ be a set of $n$ \textit{weights\/}. We begin by
considering the following left $S^1$-action on $\Gl(n)$.
\begin{equation*}
  \phi: S^1 \times \Gl(n) \to \Gl(n), \quad z \cdot A =\begin{pmatrix} z^{w_1} & & & \\ &  z^{w_2} &
    &  \\ 
    & & \ddots  & \\ & & & z^{w_n} \end{pmatrix} A
\end{equation*}

We consider initially the cohomology Serre spectral sequence arising from the fibration $ES^1
\times_{S^1} \Gl(n) \to BS^1 \isom \C P^{\infty}$, with fiber $\Gl(n)$. We write $T$ for the maximal
torus of $U(n) \subset \Gl(n)$. The inclusion $U(n) \to \Gl(n)$ is a homotopy equivalence. There is
a sequence of group homomorphisms $S^1 \overset{f}\to T \to U(n) \to \Gl(n)$, where the map $f$ is
given by $z \mapsto \diag{z^{w_1}, z^{w_2}, \dots, z^{w_n}}$. We obtain in this way a compatible
sequence of group actions on $\Gl(n)$.
\begin{equation*}
  \xymatrix{   S^1 \times \Gl(n) \ar[d] \ar[r] &  T \times \Gl(n)
    \ar[r] \ar[d] & \Gl(n) \times \Gl(n) \ar[d] \\
    \Gl(n) \ar@{=}[r] & \Gl(n) \ar@{=}[r] & \Gl(n) }
\end{equation*}
and consequently maps of fibrations, each with fiber $\Gl(n)$,
\begin{equation*}
  \xymatrix{ ES^1
    \times_{S^1} \Gl(n) \ar[r] \ar[d] & ET \times_T \Gl(n) \ar[r] \ar[d]  & E\Gl(n) \times_{\Gl(n)}
    \Gl(n) \ar[d] \\ BS^1 \ar[r] & BT \ar[r] & B\Gl(n) } 
\end{equation*}
The cohomology of $\Gl(n)$ is $H^*(\Gl(n);R) \isom \Lambda_R(\alpha_1,
\dots, \alpha_n)$, with $|\alpha_i| = 2i-1$. The cohomology of $B\Gl(n), BT$ and $BS^1$ are 
\begin{equation*}
  H^*(B\Gl(n); R) \isom R[c_1, \dots, c_n], \quad H^*(BT;R) \isom R[t_1, \dots, t_n], \quad
  H^*(BS^1;R) \isom R[t]
\end{equation*}
where $|c_i| = 2i$, and $|t_i| = |t| = 1$. The map induced by the inclusion $T \to \Gl(n)$ on the
cohomology of classifying spaces is $\psi: \Z[c_1, \dots, c_n] \to \Z[t_1, \dots, t_n]$ sending $c_i$ to
$\sigma_i(t_1, \dots, t_n)$, the $i$-the elementary symmetric function in the $t_j$. The map on the
cohomology of classifying spaces induced by $S^1 \to T$ is given by $\theta: \Z[t_1, \dots, t_n] \to
\Z[t]$, sending $t_i$ to $w_it$.

We remark that all cohomology groups we encounter are free $R$-modules, as a consequence there is
never any nontrivial behaviour in either K\"unneth or Universal Coefficient formulas.

The Serre spectral sequence arising from the $\Gl(n)$-fiber bundle $E\Gl(n) \times_{\Gl(n)}
\Gl(n) \to B\Gl(n)$ is simply the standard fibration $E\Gl(n) \to B\Gl(n)$. The action of the  differentials in
this spectral sequence is summarily described by 
\begin{equation*}
  d_{2i}(\alpha_i) = c_i, \quad d_{j}(\alpha_i) = 0 \text{ for $i \neq j$}
\end{equation*}
all other differentials can be deduced easily from these by using the product structure.

By comparison with this spectral sequence, the differentials in the Serre spectral sequence for the
fibration $ES^1 \times_{S^1} \Gl(n) \to BS^1$ can be computed. There is a comparison map,
which on the $E_2$-page is 
\begin{equation} \label{eq:leftAction1} 
  id \tensor (\theta\psi)^*: H^*(\Gl(n);R) \tensor H^*(B\Gl(n);R) \to H^*(\Gl(n);R) \tensor H^*(BS^1;R)
\end{equation}
In the latter spectral sequence, the differentials are now seen to be described by
\begin{equation} \label{eq:leftAction2} 
  d_{2i}(\alpha_i) = \sigma_i(\vec{w})t^i \pmod{\sigma_1(\vec{w})t^i, \sigma_2(\vec{w})t^i, \dots,
    \sigma_{i-1}(\vec{w})t^i}, \quad d_j(\alpha_i) = 0 \text{ for $i \neq j$} 
\end{equation}
As with the former sequence, the differentials on all other elements may be deduced from those
given.

\subsection{Action on both Left \& Right}

We now consider an action of $S^1$ on $\Gl(n)$ both on the left and on the right. If $\vec{u},
\vec{v} \in \Z^n$ are two $n$-tuples of integers, one defines an action of $S^1$ on $\Gl(n)$ by
\begin{equation*}
  z\cdot A =
  \begin{pmatrix}
    z^{u_1} & & & \\ & z^{u_2} & & \\ & & \ddots & \\ & & & z^{u_n}
  \end{pmatrix} A
  \begin{pmatrix}
    z^{-v_1} & & & \\ & z^{-v_2} & & \\ & &  \ddots & \\ & & & z^{-v_n}
  \end{pmatrix}
\end{equation*}

The argument in this section being rather long, if routine, we state it here for convenience of
reference
\begin{prop} \label{p:lrsss}
  For the $S^1$ action given above, the Serre spectral sequence of the fiber sequence $\Gl(n) \to
  ES^1 \times_{S^1} \Gl(n) \to BS^1$ has $E_2$-page
  \begin{equation*}
    H^*(\Gl(n); R) \tensor H^*(BS^1;R) \isom \EA_R(\alpha_1, \dots, \alpha_n) \tensor R[\theta].
  \end{equation*}
  The differentials in this sequence are described summarily by
  \begin{equation*}
    d_{2i}(\alpha_i) = \left[\sigma_i(\vec{u}) - \sigma_i(\vec{v})\right]\theta^i \pmod{
      \sigma_1(\vec{u}) - \sigma_1(\vec{v}) , \dots, \sigma_{i-1}(\vec{u}) - \sigma_{i-1}(\vec{v})}
  \end{equation*}
  where $i$ is any integer between $1$ and $n$; the differentials can be deduced on all other
  elements by means of the product structure.
\end{prop}

A portion of this spectral sequence is illustrated below

\begin{equation*}
  \xymatrix{ \alpha_3 & 0 &\alpha_3\theta & 0 & \alpha_3 \theta^2 \\
    \alpha_1 \alpha_2 \ar@{..>}^{d_2}[drr]  \ar@{..>}^{d_4}[dddrrrr] & 0 & \alpha_1 \alpha_2 \theta
    \ar@{..>}^{d_2}[drr] & 0 &  \alpha_1 \alpha_2 \theta^2\\ 
    \alpha_2 \ar^{d_4}[dddrrrr] & 0 & \alpha_2\theta & 0 & \alpha_2 \theta^2 \\
    0 & 0 & 0 & 0 & 0 \\
    \alpha_1 \ar^{d_2}[drr]& 0 & \alpha_1 \theta \ar@{..>}^{d_2}[drr] & 0 & \alpha_1 \theta^2 \\
    1 & 0 & \theta & 0 & \theta^2 
  }
\end{equation*}
The differentials denoted by dotted arrows can be deduced from those denoted by solid arrows. A
differential between illustrated groups that is not marked by an arrow must be $0$.

There is a group homomorphism $f:S^1 \to \Gl(n) \times \Gl(n)$ given by
\begin{equation*}
  f(z) = \left(  \begin{pmatrix}
    z^{w_1} & & & \\ & z^{w_2} & & \\ & & \ddots & \\ & & & z^{w_n}
  \end{pmatrix},   \begin{pmatrix}
    z^{v_1} & & & \\ & z^{v_2} & & \\ & &  \ddots & \\ & & & z^{v_n}
  \end{pmatrix} \right)
\end{equation*}
Let $\Gl(n) \times \Gl(n)$ act on $\Gl(n)$ by $(g,h)\cdot A = gAh^{-1}$. We write the action map as
$\alpha: (\Gl(n) \times \Gl(n)) \times \Gl(n) \to \Gl(n)$. There is a
commutative diagram
\begin{equation*}
  \xymatrix{ S^1 \times \Gl(n) \ar[d] \ar^f[rr]  & & (\Gl(n) \times \Gl(n)) \times \Gl(n) \ar^{\alpha}[d] \\
    \Gl(n) \ar@{=}[rr] & & \Gl(n)}
\end{equation*}
We concentrate for now on the action $\alpha$, and in particular, on the Serre spectral sequence
associated with the fibration $E(\Gl(n) \times \Gl(n)) \times_{\Gl(n) \times \Gl(n)} \Gl(n) \to
B(\Gl(n) \times \Gl(n))$.

In the first place we know that $B(\Gl(n) \times \Gl(n)) \weq B\Gl(n) \times B\Gl(n)$, and the
cohomology $H^*(B(\Gl(n)\times \Gl(n)); R)$ is $R[c_1, c_1', c_2, c_2',\dots, c_n, c_n']$. The
homomorphisms $\iota_1, \iota_2: \Gl(n) \to \Gl(n) \times \Gl(n)$ sending $a \mapsto (a,I_n)$ and $a
\mapsto (I_n,a)$ respectively, induce evaluations of $R[c_1, c_1', c_2, c_2',\dots, c_n, c_n']$ at
$(c_1', \dots, c_n') = 0$ and $(c_1, \dots, c_n) =0$ respectively.\footnote{The two projection homomorphisms
$\Gl(n) \times \Gl(n) \to \Gl(n)$ induce the inclusions $R[c_1, \dotsm c_n], R[c_1', \dotsm c_n']
\subset R[c_1,c_1', \dots, c_n,c_n']$, but we shall not use this fact in the sequel.}

There is a diagram
\begin{equation*}
  \xymatrix{ \Gl(n) \times \Gl(n) \ar^{\iota_1, \id}[r] \ar[d] & (\Gl(n) \times \Gl(n)) \times \Gl(n)
    \ar[d] \\ \Gl(n) \ar@{=}[r] & \Gl(n) }
\end{equation*}
where the action on the left is the usual action of $\Gl(n)$ on itself on the left. There is a map
of Serre spectral sequences, which on the $E_2$-page is
\begin{equation*}
  \xymatrix{ H^*(\Gl(n); R) \tensor H^*(B(\Gl(n) \times \Gl(n));R) \ar[d]^{\id \tensor \iota_1^*}
    \ar@{=>}[r] & H^*(E(\Gl(n) \times \Gl(n)) \times_{\Gl_n \times \Gl_n} \Gl(n); R) \ar[d] \\
     H^*(\Gl(n); R) \tensor H^*(B\Gl(n);R) \ar@{=>}[r] & H^*(E\Gl(n) \times_{\Gl_n} \Gl(n); R) }
\end{equation*}
the lower sequence of which has already been described in equations \eqref{eq:leftAction1} and
\eqref{eq:leftAction2}. By comparing the sequences, we know the generators $\alpha_i$ of
$H^*(\Gl(n);\Z)$ are transgressive in the upper sequence. General considerations imply that the
differentials in the upper sequence satisfy
\begin{equation*}
  d_{2i}(\alpha_i) = p_i(c_1,c_1', \dots, c_i, c_i') \pmod{ p_1, \dots, p_{i-1}}
\end{equation*}
where $p_i(c_1, c_1', \dots, c_i, c_i')$ is a homogeneous polynomial of degree $2i$. By comparison
with the lower sequence, we know that after evaluating at $c_1' = c_2' = \dots =c_n' = 0$, this
polynomial becomes precisely $c_i$. By degree-counting, we know that
\begin{equation} \label{eq:pidef} 
  p_i = c_i + ac_i' + q_i(c_1, c_1', c_2, c_2', \dots, c_{i-1}, c_{i-1}') \pmod{ p_1, \dots, p_{i-1}}
\end{equation}

There is an obvious involution, $\tau$, on $\Gl(n) \times \Gl(n)$ interchanging the two factors, and
an involution, $\sigma$, on $\Gl(n)$, sending $A$ to $A^{-1}$. These are compatible in the sense that the
following commutes
\begin{equation*}
  \xymatrix{ \left( \Gl(n) \times \Gl(n) \right) \times \Gl(n) \ar^{\tau, \sigma}[r] \ar[d] &
    \left( \Gl(n) \times \Gl(n) \right) \times \Gl(n) \ar[d] \\
    \Gl(n) \ar^{\sigma}[r] & \Gl(n)}
\end{equation*}
The vertical action maps are $\alpha$ both cases (this is simply the equation $hA^{-1}g^{-1} =
(gAh^{-1})^{-1}$). There is a resulting map of Serre spectral sequences associated with these group
actions, it is in this case an automorphism, which on the $E_2$-page manifests as
\begin{align*}
   H^*(\Gl(n); R) \tensor H^*(B(\Gl(n) \times \Gl(n));R) & \longrightarrow H^*(\Gl(n); R)
    \tensor H^*(B(\Gl(n) \times \Gl(n));R)   \\
    \begin{matrix} \alpha_i \tensor 1 \\  1 \tensor c_i \\ 1 \tensor c_i'\end{matrix}
    &\longmapsto
    \begin{matrix} -\alpha_i \tensor 1 \\ 1 \tensor c_i' \\ 1 \tensor c_i \end{matrix}
\end{align*}
From this and linearity of the differentials, one sees that 
\begin{equation*}
  p_i(c_1, c_1', \dots, c_i, c_i') = d_{2i}(\alpha_i) = - d_{2i}(-\alpha_i) = -p_i(c_1', c_1, \dots,
  c_i', c_i) \pmod{p_1, \dots, p_{i-1}}.
\end{equation*}
Comparing this with equation \eqref{eq:pidef}, we see that $a = -1$ and furthermore that $q_i$ also
satisfies
\begin{equation} \label{eq:qidef} 
  q_i(c_1',c_1, c_2', c_2, \dots, c_{i-1}', c_{i-1}) =  - q_i(c_1,c_1', c_2, c_2', \dots, c_{i-1},
  c_{i-1}') \pmod{ p_1, \dots, p_{i-1}}  
\end{equation}

One now sees by a simple induction that the following holds
\begin{lemma}
  In the notation of the previous paragraphs, for all $i \le n$
  \begin{equation*}
    d_{2i}(\alpha_i) = c_i - c_i' \pmod{ c_1-c_1', c_2-c_2', \dots, c_{i-1}-c_{i-1}'} 
  \end{equation*}
\end{lemma}
\begin{proof}
  This is evident for $i=1$. Suppose the result holds for $i-1$. Then, from equation
  \eqref{eq:pidef}, we have
  \begin{equation*}
    d_{2i}(\alpha_i) = c_1 - c_1' + q_i(c_1,c_1',c_2, c_2', \dots, c_{i-1},c_{i-1}') \pmod{
      c_1-c_1', c_2-c_2', \dots, c_{i-1}-c_{i-1}'}.
  \end{equation*}
  It is a basic fact that
  \begin{equation*} \begin{split}
    q_i(c_1,c_1',c_2, c_2', \dots, c_{i-1},c_{i-1}') = q_i(c_1',c_1',c_2', c_2', \dots,
    c_{i-1}',c_{i-1}') \\ \pmod{
      c_1-c_1', c_2-c_2', \dots, c_{i-1}-c_{i-1}'} \end{split}
  \end{equation*}
  but by comparison with equation \eqref{eq:qidef} we see that the left hand side above is $0$.
\end{proof}
\medskip

We now return to consideration of the action of $S^1$ on $\Gl(n)$ on the left \& right. Our basic
tool is the homomorphism $S^1 \to \Gl(n) \times \Gl(n)$. This induces a homomorphism on the
cohomology of classifying spaces
\begin{equation*}
  R[c_1, c_1', c_2, c_2', \dots, c_n,c_n'] = H^*(B(\Gl(n) \times B\Gl(n)); R) \to \Z[t] =
  H^*(BS^1; R)
\end{equation*}
which is given by $c_i \mapsto \sigma_i(\vec{w})t^i$ and $c_i' \mapsto \sigma_i(\vec{v})t^i$.

The group homomorphism being compatible with the action on $\Gl(n)$, we have a map of Serre spectral
sequences, on the $E_2$-page this manifests as
\begin{align*}
  H^*(\Gl(n); R) \tensor H^*(B(\Gl(n) \times \Gl(n));R) &\longrightarrow H^*(\Gl(n); R)
  \tensor H^*(BS^1;R) \\
  \begin{pmatrix} \alpha_i \tensor 1 \\   1 \tensor c_i \\    1 \tensor c_i'\end{pmatrix} 
  &\longmapsto  
  \begin{pmatrix} \alpha_i \tensor 1 \\ 1 \tensor \sigma_i(\vec{w})t^i \\ 1 \tensor
    \sigma_i(\vec{v})t^i \end{pmatrix}
\end{align*}
the differentials in the second sequence can be determined from this comparison, since we know
that
\begin{equation*}
  d_i(\alpha_i) = \sigma_i(\vec{v})t^i - \sigma_i(\vec{v})t^i
\pmod{ \sigma_1(\vec{v}) - \sigma_1(\vec{v}), \sigma_2(\vec{v}) - \sigma_2(\vec{v}), \dots,
  \sigma_{i-1}(\vec{v}) - \sigma_{i-1}(\vec{v})}
\end{equation*}
as claimed in proposition \ref{p:lrsss}.

\section{The Equivariant Cohomology of Stiefel Manifolds}

We shall consider the unreduced Stiefel manifolds, to wit, the spaces $W(n,m)$ of sequences of $m$
linearly independent vectors in $\C^n$. One has $W(n,n) = \Gl(n)$, of course.  One can view $W(n,m)$
either as the space of surjective maps $\C^n \to \C^m$ or as the space of injective maps $\C^m \to
\C^n$.  We adopt the former interpretation for preference. Given $m' \le n- m$, and a surjective map
$A: \C^n \to \C^{m+m'}$, we can compose with the evident projection $\C^{m+m'} \to \C^m$, and so obtain
a surjective map $A' : \C^n \to \C^m$. We therefore have a map $\pi: W(n,m+m') \to W(n,m)$. In coordinates,
this map is straightforward
\begin{equation*}
  \begin{pmatrix}
    a_{1,1} & \dots & a_{m,1} & a_{m+1,1} & \dots & a_{m+m',1} \\ 
    \vdots & & \vdots & \vdots & & \vdots \\
     a_{1,n} & \dots & a_{m,n} & a_{m+1,n} & \dots & a_{m+m',n}
  \end{pmatrix} \mapsto
\begin{pmatrix}
    a_{1,1} & \dots & a_{m,1} \\ 
    \vdots & & \vdots & \\
     a_{1,n} & \dots & a_{m,n}
  \end{pmatrix}
\end{equation*}

We now summarise what we will need of the non-equivariant cohomology of such Stiefel manifolds.
\begin{prop} \label{p:CohoStiefel}
  The cohomology of the Stiefel manifold $W(n,m)$ is given by
  \begin{equation*}
    H^*(W(n,m);\Z) \isom \EA_\Z(\alpha_{n-m+1}, \alpha_{n-m+1}, \dots, \alpha_n).
  \end{equation*}
  For $n \ge m+m'$, there exists a projection map $\pi: W(n,m+m') \to W(n,m)$ (as above) by taking the first $m$
  of $m'$ vectors. On cohomology, such a map induces an inclusion
  \begin{equation*}
    \xymatrix{   H^*(W(n,m);R) \isom  \EA_R(\alpha_{n-m+1}, \alpha_{n-m+1}, \dots, \alpha_n)
      \ar[r] &  \EA_R(\alpha_{n-m+1}, \alpha_{n-m'+1}, \dots, \alpha_n) \\
    \alpha_i \ar@{|->}[r] & \alpha_i } 
  \end{equation*}
\end{prop}

If the vector spaces are given a grading, the space $W(n,m)$ is subject to a left-right action. More
precisely, if $U,V$ are graded vector spaces of dimensions $n$ and $m$ respectively, admitting bases
of homogeneous elements in degrees $(u_1, \dots, u_n) \in \Z^n$ and $(v_1, v_2, \dots, v_m) \in
\Z^m$, and if $A: V \to W$ is a surjective map, then we can write $A$ with respect to the given
bases, and the left-right action of $S^1$ of $A$ is given by
\begin{equation*}
   z\cdot A =
  \begin{pmatrix}
    z^{w_1} & & & \\ & z^{w_2} & & \\ & & \ddots & \\ & & & z^{w_m}
  \end{pmatrix} A
  \begin{pmatrix}
    z^{-v_1} & & & \\ & z^{-v_2} & & \\ & &  \ddots & \\ & & & z^{-v_n}
  \end{pmatrix}
\end{equation*}

Given graded vector spaces $V, V'$ of dimensions $m,m'$ with $m+m' \le n$, and a surjective map $A:U
\to V \oplus V'$, we can take the composite $A: U \to V \oplus V' \to V$. This gives an
$S^1$-equivariant version of the map $\pi$ in the obvious manner.

\begin{theorem} \label{th:EqCohStiefel}
  Suppose $U,V$ are $n$, $m$-dimensional graded vector spaces, respectively, with admitting bases of
  homogeneous elements in degrees $\vec{u} =(u_1, \dots, u_n) \in \Z^n$ and $\vec{v} = (v_1, v_2,
  \dots, v_m) \in \Z^m$. Let $W(n,m) \to ES^1 \times_{S^1} W(n,m) \to BS^1$ be the fiber sequence
  associated with the induced $S^1$-action on $W(n,m)$, the space of surjective maps $U \to V$. Let
  $s_1, \dots, s_{n-m}$ be integers defined (recursively) by the relations
  \begin{equation*}
    \sigma_i(\vec{u}) = \sum_{j =1}^i \sigma_j(\vec{v})s_{i-j}.
  \end{equation*}
  and let $s_i =0$ for $i > n-m$.

  The Serre spectral sequence associated with the given fibration
  has $E_2$-page
  \begin{equation*}
    H^*(W(n,m); \Z) \tensor H^*(BS^1; \Z) = \EA_\Z(\alpha_{n-m+1}, \dots, \alpha_n) \tensor \Z[\theta]
  \end{equation*}
  Let $n-m+1 \le k \le n$ and suppose that 
  \begin{equation*}
    d_{2j}(\alpha_j) = 0 \quad \text{ for $n-m+1 \le i \le k$}
  \end{equation*}
  then we have, in the given spectral sequence
  \begin{equation} \label{eq:modprel}
    d_{2k}(\alpha_k) = \big[\sigma_k(\vec{u}) - \sum_{j=1}^k \sigma_j(\vec{v})s_{k-j}\big] \theta^k
  \end{equation}
\end{theorem}
\begin{proof}
  Strictly speaking, the proof proceeds by induction on $\ell = k-(n-m+1)$, although most of the
  difficulty is already evident in the case $\ell = 0$. The arguments for the case $\ell = 0$ and
  for the induction step are very similar; we shall give both in parallel as much as possible.
  
  Given the form of the spectral sequence, and the fact that all $\alpha_i$ are transgressive, one
  certainly has $d_{2k}(\alpha_k) = C\theta^k$.  The argument in the proof will be to reduce to
  $\Z/p$-coefficients, then verify equation \eqref{eq:modprel} for infinitely many primes $p$.  The
  case of $\Z$ coefficients then follows.  \medskip

  Consider the integer polynomial
  \begin{equation*}
    f(x) = x^{n-m+1} - s_1x^{n-m} + s_2x^{n-m-1} + \dots + (-1)^{n-m-1}s_{n-m-1}x +  (-1)^{n-m}s_{n-m}
  \end{equation*}
  By the a corollary of the Frobenius density theorem, this polynomial splits over $\Z/p$ for
  infinitely many primes $p$. Let $\mathcal{P}$ denote the set of all such primes. We will first
  establish equation \eqref{eq:modprel} for all $p \in \mathcal{P}$.
  \smallskip

  To show the problem reduces well to prime coefficients, we need only observe that the natural map
  induced by $\Z \to \Z/p$ commutes with differentials, and since all the $\Z$-modules under
  consideration are free, on the $E_2$-page it takes the form of a map
  \begin{equation*}
    H^p(BS^1; H^q(W(n,m);\Z)) \to H^p(BS^1; H^q(W(n,m);\Z/p))
  \end{equation*}
  which is always surjective.

  Working modulo a particular prime, $p \in \mathcal{P}$, we write $d_{2k}(\alpha_k) = C\theta^k$,
  by abuse of notation, where all the terms are understood as the reduction mod $p$ of their
  integral analogues. By the choice of the prime $p$, we can find $\vec{v}' = (v'_1, \dots,
  v'_{n-m})$, roots of the polynomial $f(x)$ in $Z/p$. We chose particular integer representatives
  for the $v_i'$, denoting them by $v_i'$, again by abuse of notation. We note that this choice of
  $v_i'$ has been made so that $\sigma_i(\vec{v}') \equiv s_i \pmod{p}$. We work exclusively with
  $\Z/p$ coefficients from now on. We therefore have $\sigma_i(\vec{v}') = s_i$, which is key to the
  whole argument.

  Let $V'$ be a graded vector space with homogeneous basis elements in degrees $\vec{v}' = (v_1',
  \dots, v'_{n-m})$. Let $\pi$ denote the equivariant projection map $\pi: \Gl(n) \to W(n,m)$ we
  obtain by considering the former as surjective maps $U \to V \oplus V'$, and the latter as
  surjective maps $U \to V$. This map induces a map of Serre spectral sequences.

  For the purposes of computing with the $S^1$-action on $\Gl(n)$, it will be convenient to define
  $\vec{v}''$ as the concatenation of $\vec{v}$ and $\vec{v}'$, that is $(v_1, \dots, v_m, v_1',
  \dots, v_{n-m}')$. The spectral sequence arising from the $S^1$ action on $\Gl(n)$ has $E_2$ page
  \begin{equation*}
    H^*(BS^1; \Z/p) \tensor H^*(\Gl(n);\Z/p) = \Z/p[\theta] \tensor \EA_{\Z/p}(\alpha_1, \dots, \alpha_n)
  \end{equation*}
  the first nonvanishing differential is given by $d_{2j}(\alpha_j) = [\sigma_j(\vec{u}) -
  \sigma_j(\vec{v}'')] \theta^j$, where $j$ is the least positive integer such that
  $\sigma_j(\vec{u}) - \sigma_j(\vec{v}'') \neq 0$.

  Suppose for the sake of contradiction that this $j \le n-m$.
  \begin{equation*}
    d_{2j}(\alpha_j) = \sigma_j(\vec{u}) - \sigma_j(\vec{v}'') = \sigma_j(\vec{u}) - \sum_{i=1}^j
    \sigma_i(\vec{v}) \sigma_{j-i}(\vec{v}') = \sigma_j(\vec{u}) - \sum_{i=1}^j
    \sigma_i(\vec{v}) s_{j-i}
  \end{equation*}
  The term on the \textsc{rhs} is 0, by the definition of $s_{j-i}$, a contradiction.

  If we are in the case where $k > n-m+1$, that is, not the base case for the purpose of induction,
  we suppose for the sake of contradiction that $j < k$. Again we have
  \begin{equation*}
     d_{2j}(\alpha_j) = \sigma_j(\vec{u}) - \sigma_j(\vec{v}'') = \sigma_j(\vec{u}) - \sum_{i=1}^j
    \sigma_i(\vec{v}) \sigma_{j-i}(\vec{v}') = \sigma_j(\vec{u}) - \sum_{i=1}^j
    \sigma_i(\vec{v}) s_{j-i}
  \end{equation*}
  but here we know that, by applying the result to compute $d_{2j}(\alpha_j)$, that we have
  $d_{2j}(\alpha_j) = \sigma_j(\vec{u}) - \sum_{i=1}^j \sigma_i(\vec{v}) s_{j-i}$. Since
  $d_{2j}(\alpha_j)$ was assumed to be $0$ for $j$ in the range $n-m < j < k$, we have a
  contradiction.
  
  We return to considering both cases. The purpose of the inductive argument is to be able to assert
  that $d_{2j}(\alpha_j) = 0$ for $j< k$ in the spectral sequence
  \begin{equation*}
     H^*(BS^1; \Z/p) \tensor H^*(\Gl(n);\Z/p) \Rightarrow H^*(ES^1 \times_{S^1} \Gl(n,m);\Z/p).
  \end{equation*}
  This allows us to compute the differential $d_{2k}(\alpha_k)$ in this spectral sequence. It is
  \begin{equation*}
    d_{2k}(\alpha_k) = [\sigma_k(\vec{u}) - \sigma_k(\vec{v}'')]\theta^k  = \big[ \sigma_k(\vec{u})
    - \sum_{j=1}^k \sigma_j(\vec{v})\sigma_{k-j}(\vec{v}') \big] =  \big[ \sigma_k(\vec{u})
    - \sum_{j=1}^k \sigma_j(\vec{v})s_{k-j} \big]
  \end{equation*}
  
  We again consider the cohomology of $ES^1 \times_{S^1} W(n,m)$. From the comparison map
  \begin{equation*}
    \xymatrix{ H^*(BS^1; \Z/p) \tensor H^*(W(n,m);\Z/p) \ar@{=>}[rr] \ar^{\id \tensor \pi^*}[d] &&
      H^*(ES^1 \times_{S^1} W(n,m);\Z/p) \ar[d] \\
      H^*(BS^1; \Z/p) \tensor H^*(\Gl(n);\Z/p) \ar@{=>}[rr] &&  H^*(ES^1 \times_{S^1}
      \Gl(n);\Z/p) }
  \end{equation*}
  we obtain a commutative square 
  \begin{equation*}
    \xymatrix{ \alpha_k \ar@{|->}^{d_{2k}}[rr] \ar@{|->}[d] & & C\theta^k \ar@{|->}[d] \\ 
      \alpha_k \ar@{|->}^{d_{2k}}[rr] & &  \big[\sigma_k(\vec{u}) - \sum_{j=1}^k
      \sigma_j(\vec{v})s_{k-j}\big] \theta^k}
  \end{equation*}
  which proves that $C = \big[\sigma_k(\vec{u}) - \sum_{j=1}^k \sigma_j(\vec{v})s_{k-j}\big]$ in
  $\Z/p$, or equivalently that we have 
  \begin{equation*}
    p \mid \Big(C - \big[\sigma_k(\vec{u}) - \sum_{j=1}^k
  \sigma_j(\vec{v})s_{k-j}\big]\Big)
  \end{equation*}
  
  Since there are infinitely many $p \in \mathcal{P}$, and this relation holds for them all, it
  follows that 
  \begin{equation*}
    C  =\big[\sigma_k(\vec{u}) - \sum_{j=1}^k\sigma_j(\vec{v})s_{k-j}\big]
  \end{equation*}
  as required.
\end{proof}

Unfortunately this method of proof establishes only the first non-zero differential of the form
$d_{2k}(\alpha_k) = C\theta^k$, we cannot push it further to describe the subsequent
differentials. We conjecture that the pattern established in the theorem continues, that the
differential takes the form
\begin{equation*}
  d_{2k}(\alpha_k) = \big[\sigma_k(\vec{u}) - \sum_{j=1}^k \sigma_j(\vec{v})s_{k-j}\big] \theta^k
\end{equation*}
modulo the appropriate indeterminacy for all $k$.

\section{The General Rothenberg-Steenrod Spectral Sequence} 
This section is devoted to a construction of a spectral sequence which is well-known to homotopy
theorists, but seldom published in the form we need, we establish it here for want of a good
reference. It seems to be most properly called a spectral sequence of
Eilenberg-Moore-Rothenberg-Steenrod type, after \cite{ROTHENBERGSTEENROD}; we shall abbreviate and
refer to it as the Rothenberg-Steenrod spectral sequence.

We shall be working with simplicial spaces in this section, where appropriate (e.g.~for purposes of
computing singular homology) we shall tacitly replace such a simplicial space by its geometric
realization.

Let $G$ be a topological group. Let $X$, $Y$ be pointed spaces on which $G$ acts on the right and on
the left respectively. As in \cite{MAYCLASS}, one defines $B(X,G,Y)$ as the simplicial space whose
$n$-simplices are
\begin{equation*}
  Z_n(X,G,Y) = X \times\overbrace{G \times \dots \times G}^{\text{$n-1$-times}}\times Y.
\end{equation*}
The face maps in this simplicial space are given by combining successive pairs of spaces in $Z_n$
via the action maps $X \times G \to X$, $G \times Y \to Y$ and the multiplication map $G \times G
\to G$.

As in \cite{SEGAL}, for any simplicial space $A_\bullet$, one has spectral sequences (in
loc.~cit.~only the cohomology case is handled and that for semisimplicial spaces, but the case of
homology is similar and the necessary modifications for simplicial spaces not difficult)
\begin{align*}
  E^1_{p,q} &= H_q( A_p; R) \Rightarrow H_{p+q}(A_\bullet; R) \\
  E_1^{p,q} &= H^q( A_p; R) \Rightarrow H^{p+q}(A_\bullet; R)
\end{align*}
which are natural in $A_\bullet$. In our particular case, we have spectral sequences
\begin{align*}
  E^1_{p,q} &= H_q( Z_p(X,G,Y); R) \Rightarrow H_{p+q}(B(X,G,Y); R) \\
  E_1^{p,q} &= H^q( Z_p(X,G,Y); R) \Rightarrow H^{p+q}(B(X,G,Y);R)
\end{align*}

It is proved in loc.~cit.~that the $d_1$-differential in the spectral sequences $d_1:H^q(Z_p(X,G,Y);R) \to
H^p(Z_{p+1}(X,G,Y);R)$ is the alternating sum of the maps induced on homology by the face maps of the
simplicial space. The analogous result holds in homology.

We make the simplifying assumption that $H_*(G;R)$, $H_*(X;R)$ and $H_*(Y;R)$ are projective $R$-modules,
so that a K\"unneth isomorphism obtains
\begin{equation*}
  H_*(Z_q(X,G,Y);R) = H_*(X;R) \tensor_R H_*(G;R)^{\tensor q} \tensor_R H_*(Y;R)
\end{equation*}
and one also has universal coefficient isomorphisms for $X,Y$ and $G$
\begin{equation*}
  H^*(X;R) \isom \Hom( H_*(X;R), R)
\end{equation*}
and similarly for $Y$ and $G$. 

In this case, a homology class can be represented as a sum of terms of the form $\xi \tensor
\gamma_1 \tensor \dots \tensor \gamma_q \tensor \eta$, and one has
\begin{align*}
  d_1(\xi \tensor\gamma_1 \tensor \dots \tensor \gamma_q \tensor \eta) & = \xi\gamma_1 \tensor
  \gamma_2 \tensor \dots
  \tensor \gamma_n \tensor \eta \\
  &+\sum_{i=1}^q (-1)^i \xi \tensor \gamma_1 \tensor \dots \tensor \gamma_{i-1}\tensor
  \gamma_i \gamma_{i+1} \tensor \gamma_{i+2} \tensor \dots \tensor \gamma_q \tensor \eta \\
  &+ (-1)^{q+1} \xi \tensor \gamma_1 \tensor \dots \tensor \gamma_{q-1} \tensor \gamma_{q-1}\eta
\end{align*}
The differential in the case of cohomology is dual to the differential described above.
\medskip

We now concentrate on the case of a single space, $Y$, on which $G$ acts on the left.
\begin{prop}
  There is a convergent spectral sequence 
  \begin{equation*}
    E_2^{p,q} = \Ext^{p,q}_{H_*(G;R)}(H_*(Y;R), R) \Longrightarrow H^{p+q}(B(\pt, G,Y);R)
  \end{equation*}
  which is natural in both $Y$ and $G$, in the sense that if $\phi: G \to G'$ is a group homomorphism, and
  $f: Y \to Y'$ is a map from a $G$-space to a $G'$-space so that the following diagram commutes
  \begin{equation*}
    \xymatrix{ G \times Y \ar[r]\ar[d] & G' \times Y' \ar[d] \\ Y \ar[r] & Y'}
  \end{equation*}
  then there is a map of spectral sequences
  \begin{equation*}
   \xymatrix{ E_2^{p,q} = \Ext^{p,q}_{H_*(G';R)}(H_*(Y';R), R) \ar[d] \ar@{=>}[r] &  H^{p+q}(B(\pt,
     G',Y');R)  \ar[d] \\
       E_2^{p,q} = \Ext^{p,q}_{H_*(G;R)}(H_*(Y;R), R) \ar@{=>}[r] & H^{p+q}(B(\pt, G,Y);R)}
  \end{equation*}
\end{prop}
\begin{proof}
  The sequence is exactly that described above, for the cohomology of $B(\pt,G,Y)$. The key is the
  identification of the $E_2$-page.
  
  For convenience, we denote the $R$-algebra $H_*(G;R)$ by $S$.  We know that the $E_2$ page is the
  homology of a cocomplex
  \begin{equation} \label{eq:cc}
    \xymatrix{\ar[r] &  H^*(Z_p(\pt,G,Y);R)\ar[r] & H^*(Z_{p+1}(\pt,G,Y);R) \ar[r] & }
  \end{equation}
  We have
  \begin{equation*} \begin{split}
    H^*(Z_p; R) = \Hom_R( S^{\tensor p} \tensor_R H_*(Y;R), R) \\ \isom \Hom_S( S \tensor_R S^p
    \tensor_R H_*(Y;R), R) \isom \Hom_S(H_*(Z_p(G,G,Y);R),R) \end{split}
  \end{equation*}
  The group $H_*(Z_p(G,G,Y);R)$ is the $p$-th term in the $E_1$-page of the spectral sequence
  associated with $B(G,G,Y)$. A routine but messy argument regarding differentials allows us to
  identify the cocomplex in \eqref{eq:cc} with the result of applying the functor $\Hom_S(\cdot, R)$
  to the complex
  \begin{equation} \label{eq:c} \begin{split} \xymatrix{\ar[r] & S^{\tensor{p+2}} \tensor_R H_*(Y;R) \isom
      H_*(Z_{p+1}(G,G,Y);R)\ar[r] & } \\ \xymatrix{ & S^{\tensor{p+1}} \tensor_R H_*(Y;R)\isom H_*(Z_{p}(G,G,Y);R) \ar[r] &
    } \end{split}
  \end{equation}
  This complex is exactly the algebraic bar resolution of $H_*(Y;R)$ as an $S$-module over the ring
  $R$, which is, as the name suggests is a resolution of $H_*(Y;R)$, \cite{WEIBEL}. It follows that
  the $E_2$-page of the spectral sequence, viz.~the homology of \eqref{eq:cc}, is exactly as
  claimed.

  The convergence of the spectral sequence goes by the book, \cite{BOARDMAN}, since it is
  concentrated in one quadrant of the plane.

  We remark on the naturality of the sequence. In the first place, given a map of $G$-spaces, $Y \to
  Y'$, we have a map $B(\pt, G, Y) \to B(\pt, G, Y')$, so we have naturality in $Y$. In the second,
  suppose we have a group homomorphism $G' \to G$, and a space $Y$ with a $G$ action, then we have
  an induced $G'$ action on $Y$ and a map $B(\pt, G',Y) \to B(\pt, G,Y)$. In both cases we obtain
  maps of spectral sequences, and composing these maps yields exactly the naturality claimed in the
  proposition. We remark that the maps on $E_2$-pages obtained in this way are exactly the maps
  \begin{equation*}
    \xymatrix{ \Ext^{p,q}_{H_*(G';R)}(H_*(Y';R), R) \ar[r] & \Ext^{p,q}_{H_*(G;R)}(H_*(Y;R), R)}
  \end{equation*}
  one obtains by functoriality of $\Ext$ with respect to the maps $H_*(Y;R) \to H_*(Y';R)$ and
  $H_*(G;R) \to H_*(G';R)$.
\end{proof}

One case of particular importance is the following

\begin{cor} \label{c:RSSS}
  Suppose $X$ is a space with a free $G$-action on the left, then there is a spectral sequence
  \begin{equation}
    E_2^{p,q} = \Ext^{p,q}_{H_*(G;R)}(H_*(X;R), R) \Longrightarrow H^{p+q}(X/G;R)
  \end{equation}
\end{cor}
\begin{proof}
  Under the assumption of a free $G$-action, there is a weak equivalence $X/G \weq B(\pt, X, G)$,
  and the spectral sequence follows.
\end{proof}

\section{The Cohomology of Spaces of Long Exact Sequences}

\subsection{Presentation as Homogeneous Spaces}
We consider long exact sequences of graded $\C$-vector spaces, for instance:
\begin{equation*}
  \xymatrix{ 0 \ar[r] & A_1 \ar^{d_1}[r] & B_1 \ar^{e_1}[r] & A_2 \ar^{d_2}[r] & \dots \ar^{e_{n-1}}[r] & A_n
    \ar^{d_n}[r] & B_n \ar[r] & 0}
\end{equation*}
since there are $2n$ terms to this sequence, we designate this as the \textit{even\/} case, the odd
case will be that where the last term is $A_n$. We shall treat mainly the even case in what follows,
the odd case is generally much the same and we shall try to spare the reader by proving each result
only once.

Our initial treatment does not involve the grading, and is equally true of the ungraded case. Later,
the grading will play a meaningful role.

We define $a_i = \dim_k A_i$ and $b_i = \dim_k B_i$. When the grading is not important, we denote
the space of such sequences by $X(a_1, \dots, a_n, b_1, \dots, b_n)$. 

The spaces $A_i$, $B_i$ can each be decomposed into graded parts. Write $\C(n)$ for the vector space
$\C$, placed in degree $n$. We write $A_i \isom \oplus_{j=1}^{a_n}\C(v_{i,j})$, with $v_{i,j} \le
v_{i,j+1}$. The integers $v_{i,j}$ encapsulate all the grading information of the $A_i$. We make an
equivalent definition of integers $w_{i,j}$ for the $B_i$. We will occasionally write
\begin{equation*} \begin{split}
  X(v_{1,1}, \dots, v_{1,a_1}; v_{2,1}, \dots, v_{2,a_2}; \dots; v_{n,1}, \dots, v_{n,a_n}; \\ w_{1,1},
  \dots, w_{1,b_1}; w_{2,1}, \dots, a_{2,b_2}; \dots ; w_{n,1}, \dots, w_{n,b_n}) \end{split}
\end{equation*}
in place of $X(a_1, \dots, a_n, b_1, \dots, b_n)$ when the grading is important.

In the interests of concreteness, we fix a basis of homogeneous elements for each $A_i$ and each
$B_i$, and equip each space with a complex inner-product with respect to which the given basis is
orthonormal.

The space of long exact sequences is now identified with the space of matrices representing the maps
$d_i$, $e_i$. The space of automorphisms
\begin{equation*}
  G = \Gl(A_1) \times \Gl(B_1) \times \Gl(A_2) \times \Gl(B_2) \times \dots \times  \Gl(A_n) \times \Gl(B_n)
\end{equation*}
has a left-action on the space of such seqences by 
\begin{equation} \label{eq:XasGK}  \begin{split}
  (a_1, b_1, \dots, a_n, b_n)\cdot (d_1, e_1, d_2, e_2, \dots, d_n, e_n) \\= ( b_1d_1a_1^{-1},
  a_2e_1b_1^{-1}, \dots,a_ne_{n-1}b_{n-1}^{-1},  b_nd_na_n^{-1}) \end{split}
\end{equation}
This left-action is readily seen to be transitive. 
We will describe the space $X$ as a set of left cosets of $G$, and use this description to calculate the
cohomology of $X$. To do this, we fix notation

We write the homology of $\Gl(A_i)$ as $\EA_R(\d\alpha_{i,1}, \dots, \d\alpha_{i,a_i})$ and that of
$\Gl(B_i)$ as $\EA_R(\d\beta_{i,1}, \dots, \d\beta_{i,b_i})$. We then have
\begin{equation} \label{presHG} 
  H_*(G;R) = \prod_{i=1}^n \EA_R(\d\alpha_{i,1}, \dots, \d\alpha_{i,a_i}) \times \prod_{i=1}^n
  \EA_R(\d\beta_{i,1},\dots, \d\beta_{i,b_i}).
\end{equation}
At times it is more convenient to distinguish the two families of spaces $A_i$ and $B_i$, and at
other times it is more convenient to view them as being of a kind. We will occasionally therefore
use the notation
\begin{equation*}
  \alpha_{i,j} = \gamma_{2i-1,j}, \quad \beta_{i,j} = \gamma_{2i,j}, \quad a_i = c_{2i-1}, \quad b_i
  = c_{2j}.
\end{equation*}

In considering $X$ as a homogeneous $G$-space, the stabilizer of a point can be computed without too
much difficulty. Denote by $\{x_{i,j}\}$ the $j$-th element in our basis for $A_i$, and by
$\{y_{i,j}\}$ the $j$-th element in our basis for $B_i$, and by $r_i$, $s_i$ the ranks of $d_i$,
$e_i$ respectively. Consider the sequence $\xi_0$ for which
\begin{align}
  \label{seqxi0}
  d_i(x_{i,j}) &= y_{i,j}, \quad \text{ for $1 \le j \le r_i$}, \qquad d_i(x_{i,j}) = 0,
  \quad \text{ for $j> r_i$} \\
  e_i(y_{i,b_i-j}) &= x_{i+1, a_{i+1}-j}, \quad \text{ for $1 \le j \le s_i$}, \qquad
  e_i(y_{i,b_i-j})= 0 \quad \text{ for $j > s_i$}
\end{align}
We take this sequence as an origin for the homogeneous space $X$ whenever we require such a point.

Supposing  $(a_1, b_1, \dots, a_n, b_n)$ fixes this sequence, then we have, in particular, the
equations $b_i d_i a_i^{-1} = d_i$ and $a_{i+1}e_ib_i^{-1}$. In coordinates
\begin{equation*}
  b_i
  \begin{pmatrix}
    I_{r_i} & 0 \\ 0 & 0 
  \end{pmatrix}
  a_i^{-1} =   \begin{pmatrix}
    I_{r_i} & 0 \\ 0 & 0 
  \end{pmatrix}, \quad  a_{i+1}
  \begin{pmatrix}
    I_{s_i} & 0 \\ 0 & 0 
  \end{pmatrix}
  b_i^{-1} =   \begin{pmatrix}
    0 & 0 \\ 0 & I_{s_i}
  \end{pmatrix}
\end{equation*}
From these equations it follows that the matrices $a_i$ and $b_i$ decompose as
\begin{equation*}
  b_i =
  \begin{pmatrix}
    f_i & * \\ 0 & g_i
  \end{pmatrix}, \quad  a_i = \begin{pmatrix}
    f_i & * \\ 0 & g_{i-1}
\end{pmatrix}
\end{equation*}

We see that the stabiliser of an arbitrary element is therefore a group having the homotopy type of
\begin{equation*}
  K= \Gl(r_1) \times \Gl(s_1) \times \Gl(r_2) \times \Gl(s_2) \times \dots \times \Gl(s_{n-1}) \times \Gl(r_n)
\end{equation*}
We write the homology of this space as
\begin{equation*}
  H_*(K;R) = \prod_{i=1}^n \EA_R(\d \rho_{i,1}, \dots, \d\rho_{i,r_i}) \times \prod_{i=1}^{n-1}
  \EA_R(\d \sigma_{i,1}, \d\sigma_{i,2}, \dots, \d \sigma_{i,s_i})
\end{equation*}
but again, in keeping with our dual understanding of $A_i$, $B_i$ we will write
\begin{equation*}
  \d \rho_{i,j} = \d \tau_{2i-1,j}, \quad \d \sigma_{i,j} = \d \tau_{2i,j}, \quad r_i = t_{2i-1},
  \quad s_i = t_{2i}
\end{equation*}
we shall also use the convention $t_0=0$, so that for all $c_i$ there is a corresponding $t_{i-1}$.

\subsection{The Non-Equivariant Cohomology}

The inclusion of $K$ in $G$ induces the following map on homology
\begin{equation} \label{HGasHK} 
  \xymatrix{ H_*(K;R) \ar^{\iota_*}[r] \ar@{=}[d] &  H_*( G;R) \ar@{=}[d] \\
    \prod_{i=1}^{2n} \EA_{R}(\d\tau_{i,1}, \dots, \d\tau_{i,t_i}) \ar^{\iota_*}[r] &
    \prod_{i=1}^{2n} \EA_R(\d\gamma_{i,1},  \dots, \d\gamma_{i,c_i}) \\
   \iota_*( \d\tau_{i,j}) \ar@{=}[r] & \d\gamma_{i-1,j} + \d\gamma_{i,j}}
\end{equation}
which is a map of Hopf algebras, in particular of rings. By a change of coordinates, replacing
$\d\gamma_{i,j}$ by $\iota_*(\d\tau_{i,j})$ in our presentation of $H_*(G;R)$ \eqref{presHG}, we see
that $H_*(G;R)$ is itself an exterior algebra over $H_*(K;R)$. We find a set $\d N$ so that
$H_*(G;R) = \EA_{H_*(K;R)}(\d N)$.

Let $i$ be an integer satisfying $1 \le i \le 2n$. Suppose $j$ is an integer satisfying $t_{i-1}
< j \le c_i$. We define $\ell$ to be the least integer $i \le \ell$ such that $t_\ell < j$
\begin{equation} \label{eq:Nu} 
  \kappa_{i,j} = \sum_{k=i}^\ell (-1)^{k-1}\gamma_{k,j}.
\end{equation}

We denote the set of all $\kappa_{i,j}$ by $N(X)$, or by $N$ when the dependence on $X$ is clear. By $\d N$, we mean the set of duals of $N$, taken with respect to the evident basis of $H_*(G;R) = \EA_R(\{\gamma_{i,j}\})$.

A pictorial description of the $\kappa_{i,j}$ is perhaps of use. Take for example the case where $(c_1, c_2, c_3, c_4,c_5,c_6) = (1,4,5,4,3,1)$ and $(t_1, t_2,t_3,t_4,t_5) = (1,3,2,2,1)$. Pictorially we denote this as
\begin{equation} \label{eq:exofN} 
  \xymatrix{ & &  \gamma_{3,5} & & & \\
    &  \gamma_{2,4} & \gamma_{3,4} &  \gamma_{4,4}&  \\
    & \gamma_{2,3} \ar@{-}[r] & \gamma_{3,3} & \gamma_{4,3} & \gamma_{5,3}\\ 
     & \gamma_{2,2} \ar@{-}[r] & \gamma_{3,2} \ar@{-}[r] &  \gamma_{4,2} \ar@{-}[r] & \gamma_{5,3}\\
    \gamma_{1,1} \ar@{-}[r]  & \gamma_{2,1} \ar@{-}[r] & \gamma_{3,1} \ar@{-}[r] &
    \gamma_{4,1} \ar@{-}[r]& \gamma_{5,1} \ar@{-}[r] & \gamma_{6,1} }
\end{equation}
the horizontal lines being numbered by the $t_i$. In this case one has
\begin{align*}
  \kappa_{1,1} &= \gamma_{1,1} + \gamma_{2,1} +  \gamma_{3,1} + \gamma_{4,1} +
  \gamma_{5,1}+ \gamma_{6,1}\\
  \kappa_{2,2} &= \gamma_{2,2} +  \gamma_{3,2} +  \gamma_{4,2} + \gamma_{5,2} \\
  \kappa_{2,3} &= \gamma_{2,3} +  \gamma_{3,3} \\
  \kappa_{i,j} &= \gamma_{i,j} \quad \text{ if $(i,j)$ is any of $(4,3)$, $(5,3)$, $(2,4)$,
    $(3,4)$, $(4,4)$ or $(5,3)$}
\end{align*}

\begin{prop}
  In the notation of this section, the cohomology of \[X= X(a_1, \dots, a_n, b_1, \dots, b_n)\] is
  \begin{equation*}
    \Hom_{H_*(K;R)}(H_*(G;R),R) 
  \end{equation*}
  The map $G \to G/K = X$ induces the evident injective map
  \begin{equation*}
    \xymatrix{ H^*(X;R) \isom \Hom_{H_*(K;R)}(H_*(G;R), R)  \ar@{^{(}->}[r] & \Hom_{R}(H_*(G;R),R)
      \isom H^*(G;R)} 
  \end{equation*}
  and we can identify $H^*(X;R) = \EA_R(N) \subset H^*(G;R)$.
\end{prop}
\begin{proof}
  Since $X$ is the homogeneous space of cosets of $G$ by $K$, we have a spectral sequence as in
  corollary \ref{c:RSSS}, which on the $E_2$-page is $\Ext^{p,q}_{H_*(K;R)}(H_*(G;R),R)$. The
  $H_*(K;R)$-module structure of $H_*(G;R)$ is given by \eqref{HGasHK}, $H_*(G;R)$ is an exterior
  algebra over $H_*(K;R)$, in particular it is a free $H_*(K;R)$-module. The $E_2$-page is
  concentrated in the column $p=0$, and it collapses thereafter. We therefore have $H^*(X;R) =
  \Hom_{H_*(K;R)}(H_*(G;R),R)$, which is the first assertion. 
  
  We can calculate the comparison map $H_*(G/K;R) \to H_*(G;R)$ by taking $G$ to be $G/\{e\}$, the
  trivial quotient, and using the naturality of the Rothenberg-Steenrod spectral sequence. The map
  is evidently an injection for algebraic reasons.

  To prove the last assertion, that $H^*(X;R) = \EA_R(N) \subset H^*(G;R)$, we first remark that since the
  $\kappa_{i,j}$ are $R$-independent and satisfy $\kappa_{i.j}^2=0$, the subalgebra of $H^*(G;R)$ they
  generate is exactly $\EA_R(N)$.

  Secondly, it is easily verified that the $\kappa_{i,j}$ are in fact $H_*(K;R)$-linear, so we
  certainly have an inclusion $\EA_{R}(N) \subset H^*(X;R)$. To prove the equality, we first observe
  that it suffices to deal with the case $R=\Z$, since the case of general $R$ can be obtained by
  base change.

  Consider therefore the short exact sequence
  \begin{equation*}
    \xymatrix{ 0 \ar[r] &\EA_\Z(N) \ar[r] & H^*(X;\Z) \ar[r] & Q \ar[r] & 0}
  \end{equation*}
  It suffices to show $Q=0$. Since $Q$ is a finitely generated abelian group, we need only show $Q
  \tensor \Q = 0$ and $Q \tensor \Z/p = 0$ for all $0$. After changing base to any field, $k$, we have
  \begin{equation*}
     \xymatrix{ 0 \ar[r] &\EA_k(N) \ar[r] & H^*(X;k) \ar[r] & Q \tensor_\Z k \ar[r] & 0}
  \end{equation*}
  We show that $Q \tensor_\Z k$ is $0$ by counting dimensions. In the first place we have $\dim_k
  \EA_k(N) = 2^{|N|}$. In the second we have, writing $S$ for $H_*(K;k)$ for brevity
  \begin{equation*} \begin{split}
    \dim_k(H^*(X;k) = \dim_k \Hom_{S}(H_*(G;k), k) = \dim_k \Hom_{S}(\EA_{S}(\d N), k)
    = \\ \dim_k \Hom_{S}(S^{2^{|\d N|}},k) = 2^{|\d N|} \end{split}
  \end{equation*}
  It follows that $Q \tensor_\Z k= 0$, as claimed.
\end{proof}

\subsection{The $S^1$-action}

There is an $S^1$-action on $G=\Gl(A_1) \times \Gl(B_1) \times \dots \times \Gl(A_n) \times
\Gl(B_n)$, since these vector spaces are graded and we can give $\Gl(A_i)$ an action of $S^1$ on the
left, following section \ref{SecGRVS}. This gives rise to an action of $S^1$ on $G/K = X$, which is
of central importance. To give explicit formulas, we take $\xi \in X$ to be of given by $g\xi_0$,
where $\xi_0$ is as defined in equation \eqref{seqxi0}, and $g$ is determined only up to
indeterminacy by $K$. For the sake of concreteness, take $d_1: A_1 \to B_1$, one of the
differentials in the sequence $x$. We can write this as $b_1 d_1^0a_1^{-1} = d_1$, where $a_1, b_1$
are terms in $g$. The $S^1$ action on $d_1$ is therefore given by 
\begin{equation*} \begin{split}
  z \cdot d_1 = (z\cdot b_1) d_1^0 (z \cdot a_1)^{-1} = \\ \begin{pmatrix}
    z^{w_{1,1}} & & & \\ & z^{w_{1,2}} & & \\ & & \ddots & \\ & & & z^{w_{1,b_1}}
  \end{pmatrix} d_1
  \begin{pmatrix}
    z^{-v_{1,1}} & & & \\ & z^{-v_{1,2}} & & \\ & &  \ddots & \\ & & & z^{-v_{1,a_1}}
  \end{pmatrix} \end{split}
\end{equation*}
viz.~it is an action on the left \& right as described in section \ref{SecGRVS}.

\section{Comparison Maps of Spaces of Exact Sequences}
\subsection{Comparison By Folding}
\label{ss:CompByFold}

We begin with a sequence of graded complex vector spaces
\begin{equation} \label{eq:theseqinX} 
  \xymatrix{ 0 \ar[r] & A_1 \ar^{d_1}[r] & B_1 \ar^{e_1}[r] & A_2 \ar^{d_2}[r] & \dots \ar[r]&
    B_{n-1} \ar^{e_{n-1}}[r] & A_n
    \ar^{d_n}[r] & B_n \ar[r] & 0}
\end{equation}
We continue to assume that each vector space is equipped with an inner product and an orthonormal
basis of homogeneous elements, this makes it possible to identify $\d{A}_i$ and $A_i$ as graded
vector spaces, as in section \ref{SecGRVS}, and so obtain a complex
\begin{equation*}
  \xymatrix{ 0 \ar[r] & B_1 \ar^{e_1\oplus \d d_1}[rr] & &
    A_1 \oplus A_2 \ar^{0 \oplus d_2}[r] & B_2 \ar[r] & \dots \ar[r] & A_n \ar^{d_n}[r]  & B_n \ar[r] & 0}
\end{equation*}
The map $\d{d}_1$ is injective into $A_1 \isom \d A_1$. If $e_1(b) = 0$, then $b \in \im d_1$, by
exactness, and we can write $b= d_1a$. We have $\d d_1 d_1 a \neq 0$, by nondegeneracy. It follows that
$\ker(e_{1}) \cap \ker(\d d_1) = \{0\}$, and so $e_{n-1} \oplus \d d_n$ is an injective map. The
complex is consequently an exact sequence in $X(a_1, \dots, a_n, b_1, \dots, b_{n-1}+b_n)$, and we
have a map
\begin{equation*}
  \psi:X = X(a_1, \dots, a_n, b_1, \dots, b_n) \to X(a_1, \dots, a_n, b_1, \dots, b_{n-1}+b_n) = X'
\end{equation*}
From the graded point of view, the latter space above is
\begin{equation*} \begin{split} X(\overbrace{v_{1,1}, \dots, v_{1,a_1}, v_{2,1}, \dots,
      v_{2,a_2}}^{\text{one vector of gradings}}; \dots; v_{n,1}, \dots, v_{n,a_n}; w_{1,1}, \dots,
    w_{1,b_1}; w_{2,1}, \dots, a_{2,b_2}; \dots ;\\ w_{n-1,1}, \dots, w_{n-1,b_{n-1}}; w_{n,1},
    \dots, w_{n,b_n})
  \end{split}
\end{equation*}
and the map $\psi$ is $S^1$-equivariant map by the same argument used in proposition \ref{foldingequivprec}.

There is a map $\phi: G = \Gl(A_1, B_1, \dots, B_{n-1}, A_n, B_n) \to G'= \Gl(A_1, B_1, \dots, B_{n-1} \oplus \d
B_n , A_n)$, which restricts to the composition $\Gl(B_{n-1}) \times \Gl(B_n) \to \Gl(B_{n-1})
\times \Gl(\d B_n) \to \Gl(B_{n-1} \oplus \d B_n)$. This map lifts $\psi$
\begin{equation*}
  \xymatrix{ K \ar[r] \ar@{-->}[d] &G \ar[r] \ar[d] & X(a_1, \dots, a_n, b_1, \dots, b_n) \ar^\psi[d] \\
    K' \ar[r] &G' \ar[r] & X(a_1, \dots, a_n, b_1,\dots,b_{n-1} + b_n)}
\end{equation*}
the image of the isotropy subgroup $K$ in $G'$ lies in $K'$, so one has a map of group actions on
spaces
\begin{equation*}
  \xymatrix{ G \times K \ar[r] \ar[d] & G' \times K' \ar[d] \\ G \ar[r] & G'}
\end{equation*}
by naturality in proposition \ref{c:RSSS} we have a comparison map of spectral sequences, which allows
us to compute the map $\psi^*$ on cohomology.

\begin{prop} \label{p:cbyfolding}
 We adopt the notation of this section so far. Writing
\begin{equation*}
  H^*(X;\Z)= \EA_\Z(N(X)) \quad \text{ and } H^*(X';\Z) = \EA_\Z(N(X'))
\end{equation*}
where $N(X)$ is as defined after equation \eqref{eq:Nu}. We write $N = N(X)$ and $N'=N(X')$. The map
\begin{equation*}
  \psi^*: \EA_\Z(N') \to \EA_\Z(N) 
\end{equation*}
is defined by 
\begin{equation*}
  \kappa_{1,j} \mapsto \kappa_{1,j} + \kappa_{2,j} + \kappa_{3,j}, \qquad \kappa_{i,j} \mapsto \kappa_{i+1,j}
  \text{ for $i\ge 2$}
\end{equation*}
where we employ the convention that $\kappa_{i,j} = 0$ if it is not otherwise defined.
\end{prop}
We remark that, because of exactness, $\kappa_{1,j}$ and $\kappa_{2,j}$ are never simultaneously defined.
\begin{proof}
	The maps on homology induced by $G \to G'$ and $K \to K'$ are easily computed, c.f.~section
        \ref{s:HoCoHoGln}, since these maps restrict to matrix direct-sum. We can then appeal to
        naturality in corollary \ref{c:RSSS} to obtain the result. The details are not particularly
        enlightening, and we suppress them. 
\end{proof}

\subsection{The Comparison With $W(b_1,a_1)$}

In order to compute the $S^1$-equivariant cohomology of $X$, at least to the extent of computing the
first nonzero differential in the Serre spectral sequence associated with $X \to ES^1 \times_{S^1} X
\to BS^1$
\begin{equation*}
  E_2 = R[\theta] \tensor \EA_{R}(\{\kappa_{i,j}\}) \Longrightarrow H^*(ES_1 \times_{S^1} X;R)
\end{equation*}
we use a comparison with a Stiefel manifold $W(b_1,a_1)$. This comparison is
straightforward, one takes the sequence \eqref{eq:theseqinX}, and forgets all but the differential
$A_1 \overset{d_1}{\longrightarrow} B_1$. This is represented by a $b_1 \times a_1$ matrix of rank
$a_1$. The map $X \to W(b_1,a_1)$ is evidently $S^1$-equivariant.

By a straightforward comparison argument, using naturality in corollary \ref{c:RSSS}, we obtain
\begin{prop}
	The induced map on cohomology arising from the projection $X \to W(b_1,a_1)$ is
	\begin{equation*} \small
		\xymatrix{ \EA_R(\alpha_{b_1-a_1+1}, \alpha_{b_1-a_1+2}, \dots, \alpha_{b_1}) \isom H^*(W(b_1,a_1);R) 
			\ar[rr] & & H^*(X;R) \isom \EA_R(N) \\
			\alpha_{j} \ar@{|->}[rr] && \kappa_{2,j}}
	\end{equation*}
\end{prop}

This comparison, in tandem with the folding comparison, gives a general procedure for computing many
of the differentials in the Serre spectral sequence of the fibration $X \to ES^1_{S^1} X \to
BS^1$. We say a class $\kappa_{i,j}$ \textit{supports a nonzero differential} if
$d_{2j}(\kappa_{i,j}) \neq 0$ in the aforementioned spectral sequence. The general outline of the
method is to use comparison with $W(b_1,a_1)$ to compute the differentials supported by
$\kappa_{2,b_1-a_1+i}$ for $1 \le i$, at least as far as the first such class to support a nonzero
differential. Then, one applies comparison by folding in order to extend the reach of this method to
differentials supported by $\kappa_{2,j} + \kappa_{3,j}$ for certain $j$, and so forth. The
difficulty in deducing all the differentials currently lies in theorem \ref{th:EqCohStiefel}, in
that the method for computing the differentials in the sequence associated with a Stiefel manifold
does not work past the first nonvanishing differential. Since this is the impediment, however, we
may always compute differentials $d_{2j}(\kappa_{i,j})$ if no $\kappa_{i,j'}$ for $j'<j$ supports a
nonzero differential.

\section{Herzog-K\"uhl Equations}

We return to considering chain complexes over the ring $\C[x_1, \dots, x_m] = S$. We treat of the
case of even length, the case of odd length being similar. We begin with a chain complex of graded
free $S$-modules
\begin{equation*}
  \xymatrix{ \Theta: \, 0 \ar[r] & F_{1} \ar^{D_1}[r] & G_{1} \ar^{E_1}[r] & \dots \ar^{D_n}[r] & G_{n} \ar[r] & 0}
\end{equation*}
having artinian homology. We denote the rank of $D_j$, $E_j$ by $r_j$, $s_j$ respectively.

For reasons of expediency we fix homogeneous bases of each of the free $S$-modules appearing in the
above resolution. We denote by $v_{j,1}, \dots, v_{j,a_j}$ the degrees of the basis of $F_{2j-1}$,
and by $w_{j,1}, \dots, w_{j,b_j}$ the degrees of a basis of $F_{2j}$.  With respect to these bases,
each differential in the resolution takes the form of a matrix over $S$. Since the resolution is
graded, it follows that the entries in these matrices are themselves homogeneous polynomials of
(possibly) varying degrees.

The main result of this section is the following.
\begin{thm}[Herzog-K\"uhl Equations] \label{t:HKE}
  Let $\Theta$, $v_{j,k}$ and $w_{j,k}$ be as defined above (a similar result holds for a complex of
  odd length). Let $q \le n$ and consider $r_q = \rank D_q$ (a similar result holds with the roles
  of $F_j$, $G_j$ reversed). Let $\vec{v}_q$, $\vec{w}_q$ denote the vectors of integers
  $(v_{j,k}),\, (w_{j,k})$ for $j \le q$. Let $u_1, \dots, u_q$ be integers defined recursively by
  \begin{equation*}
    \sigma_i(\vec{w}_q) = \sum_{j =1}^i \sigma_j(\vec{v}_q)u_{i-j}
  \end{equation*}
  for $1 \le i \le r_q$. Let $u_j = 0$ for $j > q$. Then we have
  \begin{equation*}
    \sigma_i(\vec{w}_q) = \sum_{j=1}^i \sigma_j(\vec{v}_q)u_{i-j}
  \end{equation*}
  for $r_q+1 \le i \le m-1$.
\end{thm}
\begin{proof}
It follows from the exactness result of \cite{CARLSSON2} that evaluation at all prime ideals of $S$
other than the irrelevant ideal yields a long exact sequence of graded vector spaces
\begin{equation*}
	\xymatrix{ 0 \ar[r] & \C^{a_1} \ar[r] & \C^{b_1} \ar[r] & \dots & \ar[r] & \C^{b_n} \ar[r] & 0 }
\end{equation*}
in particular, the matrices representing the differentials yields a map 
\begin{equation*}
f:\C^m\sm 0 \to X(a_1, \dots, a_n; b_1, \dots, b_n)= X
\end{equation*}
given by homogeneous polynomial functions. The homogeneity of the polynomials implies that this map
$f$ is $S^1$ equivariant, where $S^1 \subset \C^*$ acts on $\C^m\sm 0$ in the obvious way and on $X$
according to the grading on the various $F_{2i-1} \isom S^{a_i}$, $F_{2i} \isom S^{b_i}$. Properly
$X$ is to be understood as
\begin{equation*}\begin{split}
   X =X(v_{1,1}, \dots, v_{1,a_1}; v_{2,1}, \dots, v_{2,a_2},;\dots; v_{n,1}, \dots, v_{n,a_n}; \\ w_{1,1},
  \dots,w_{1,b_1}; w_{2,1}, \dots, w_{2,b_2}; \dots , w_{n,1}, \dots, w_{n,b_n}) \end{split}
\end{equation*}
but we will try to avoid this notation as far as possible.

Since we understand the equivariant cohomology of $X$ reasonably well, we shall try to obtain
obstructions to $S^1$-equivariant maps $\C^m\sm 0 \to X$.

We consider the equivariant composition:
\begin{equation*}
\xymatrix{ \C^m -0 \ar[r] & X \ar[r] & \pt }
\end{equation*}
There is a composition of maps of Serre spectral sequences arising from this, each fibration having
base $BS^1$ and the fibres varying. The Serre spectral sequence for $\C^m\sm 0$ has $E_2$-page:
\begin{equation*}
  H^*(\C^m\sm 0;R) \tensor_R H^*(BS^1;R) \isom \frac{R[\alpha, \theta]}{(\alpha^2)}, \qquad |\alpha| = (0,2m-1),\, |\theta| = (2,0)
\end{equation*}

The sole nontrivial differential is $d_{2m}(\alpha) = \theta^m$. 

The spectral sequence computing the equivariant cohomology of $\pt$ is even more straightforward,
which is precisely $H^*(BS^1;R) = R[\theta]$ and is in an immediate state of collapse.

The point is that the comparison map $R[\theta] \isom H_{S^1}^* (\pt; R) \to H_{S^1}^*(\C^m\sm 0;R)
\isom R[\theta]/(\theta^{2m})$ is induced by comparison maps of spectral sequences. In particular,
we obtain, on every page subsequent to the $E_2$-page, for $j < m$ and all $\ell > 1$
\begin{equation*}
  \xymatrix{ R \isom E^{2j,0}_\ell (ES^1 \times_{S^1} \pt) \ar^{\isom}[r] & E^{2j,0}_\ell (ES^1
    \times_{S^1} \C^m\sm 0) \isom R} 
\end{equation*}
Since this map must factor through $\C^m\sm 0 \to X$, taking $R = \Q$, we see that the spectral
sequence computing $H^*_{S^1}( X;\Q)$ must have $E^{2j,0}_\ell \isom \Q$. The differentials
$d_{2j}(\kappa_{i,j})$ must all vanish for $j < m$. We now apply comparison theorems to deduce the
nature of these differentials.

We first apply the comparison-by-folding map, proposition \ref{p:cbyfolding}, $2q$-times. We have a
comparison
\begin{equation*} \begin{split}
  X(a_1,a_2,\dots, a_n; b_1, b_2, \dots, b_n) \to \\ X(a_1 + a_2 + \dots + a_q, a_{q+1}, \dots, a_n;
  b_1 + b_2 + \dots + b_q, b_{q+1}, \dots, b_n) \end{split}
\end{equation*}
To this we apply a comparison with a Stiefel manifold
\begin{equation*} \begin{split}
  X(a_1 + a_2 + \dots + a_q, a_{q+1}, \dots, a_n;  b_1 + b_2 + \dots + b_q, b_{q+1}, \dots, b_n) \to \\
  W( b_1 + \dots + b_q,a_1 + \dots + a_q) \end{split}
\end{equation*}
Write $b = \sum_{i=1}^q b_i$, $a = \sum_{i=1}^q a_i$. We point out that by an exactness argument, we
have $b-a = r_q$.  The map induced on cohomology by the composite of the two comparison maps takes
$\alpha_k \in \EA_R(\alpha_{b-a+1}, \dots, \alpha_b)$ to $\sum_{i=1}^{2q} \kappa_{i,k}$. Since the
differentials in the Serre spectral sequence supported by the classes $\kappa_{i,k}$ vanish when $k
< m$, and since the comparison maps are $S^1$-equivariant, it follows that the differentials
supported by the classes $\alpha_k$, where $k<m$, in the cohomology of $W(b,a)$ similarly
vanish. The $S^1$-action on $W(b,a)$ which makes the comparison maps equivariant is given by the
weights $(v_{1,1}, v_{1,2}, \dots, v_{q,a_q}; w_{1,1}, w_{1,2}, \dots, w_{q,b_q})$. The numerical
formulas of the proposition now follow by considering theorem \ref{th:EqCohStiefel}.
\end{proof}

The numerical conditions of this theorem may be restated in the following form. Given $(v_{j,k})$
and $(w_{j,k})$ as in the theorem, there exist complex numbers $v_{q+1,1}',
\dots, v_{q+1,r_q}$ (unique up to permutation) so that, taking $\vec{v}'_q$ to be the vector of all $v_{j,k}$ and $v_{q+1,k}'$,
we have
\begin{equation*}
  \sigma_i(\vec{v}'_q) = \sigma_i(\vec{w}_q)
\end{equation*}
The $u_{j}$ of the theorem are the elementary symmetric functions $\sigma_j(\{v_{q+1,1}', \dots,
v_{q+1,r_q}'\})$. The theory of symmetric polynomials now allow us to restate the theorem as
follows
\begin{cor}
   Let $\Theta$, $v_{j,k}$ and $w_{j,k}$ be as defined above (a similar result holds for a complex of
  odd length). Let $q \le n$ and consider $r_q = \rank D_q$ (a similar result holds with the roles
  of $F_j$, $G_j$ reversed). Let $\vec{v}_q$, $\vec{w}_q$ denote the vectors of integers
  $(v_{j,k}),\, (w_{j,k})$ for $j \le q$. Let $v_{q+1,1}', \dots, v_{q+1,r_q}'$ be complex numbers
  defined by the system of equations
  \begin{equation*}
    \sum_{j=1}^q \sum_{k=1}^{b_j} w_{j,k}^i = \sum_{j=1}^q \sum_{k=1}^{a_j} v_{j,k}^i +
    \sum_{k=1}^{r_q} (v_{q+1,k}')^i
  \end{equation*}
  for $1 \le i \le r_q$. Then we have
  \begin{equation*}
    \sum_{j=1}^q \sum_{k=1}^{b_j} w_{j,k}^i = \sum_{j=1}^q \sum_{k=1}^{a_j} v_{j,k}^i +
    \sum_{k=1}^{r_q} (v_{q+1,k}')^i
  \end{equation*}
  for $r_q+1 \le i \le m-1$.
\end{cor}

The classical Herzog-K\"uhl equations correspond to the case where $q=n$, in that case $r_q=0$ and
there are no integers $u_j$ (and therefore no $v_{q+1,j}'$) to be considered. In this case we simply
have the following
\begin{cor}
   Let $\Theta$, $v_{j,k}$ and $w_{j,k}$ be as defined above (a similar result holds for a complex of
  odd length). We have
  \begin{equation*}
    \sum_{j=1}^n \sum_{k=1}^{b_j} w_{j,k}^i = \sum_{j=1}^n \sum_{k=1}^{a_j} v_{j,k}^i
  \end{equation*}
  for $0 \le i \le n$.
\end{cor}

In cases where $r_q \ge m$, of course, all the above statements are vacuous. When the complex under
consideration is a resolution, the strong Buchsbaum-Eisenbud conjecture says that the rank $r_q \ge
\binom{m-1}{2q-1}$. Since this generally exceeds $m-1$, the more intricate relations above are
conjectured not to matter for resolutions. For arbitrary complexes, however, the author does not
know of a conjecture or result that eliminates them from consideration.

\bibliographystyle{alpha}

\end{document}